\documentclass{article}
\usepackage{graphicx} 
\usepackage{a4wide}
\usepackage{color}
\usepackage{dsfont}
\usepackage{amsmath,amsfonts,amssymb,latexsym,amsthm}
\usepackage{comment,hyperref} 

\usepackage{caption} 

\newtheorem{lemma}{Lemma}

\newtheorem{theorem}{Theorem}
\newtheorem{corollary}{Corollary}
\newtheorem{proposition}{Proposition}
\title{Dynamics of roots of randomized derivative polynomials}
\author{Andr\'e Galligo and Joseph Najnudel}
\date{}

\begin{document}
\maketitle

\begin{abstract}
In this paper, we study the asymptotic macroscopic behavior of the root sets of iterated, randomized  derivatives of polynomials. 
The randomization depend on a parameter of inverse temperature $\beta \in (0, \infty]$, the case $\beta =  \infty$ corresponding to the situation where one considers the derivative of polynomials, without randomization. Our constructions can be connected to random matrix theory: in particular, as detailed in Section 2, for $\beta = 2$ and roots on the real line, we get the distribution of the eigenvalues of minors of unitarily invariant random matrices.  More detail on the connection between our setting and random matrix theory can be found, for example, in \cite{AssiotisNajnudel, BEY, BNR,Cuenca, KN,NV}. 
We prove that the asymptotic macroscopic behavior of the roots, i.e. the hydrodynamic limit, does not depend on $\beta$, and coincides with what we obtain for the non-randomized iterated derivatives, i.e. for $\beta = \infty$.  
Since recent results obtained for iterated derivations show  that the limiting dynamics is governed by a non-local and non-linear PDE, we can transfer this information to the macroscopic behavior of the  
randomized setting. Our proof is completely explicit and relies on the analysis of increments in a triangular bivariate Markov chain.

\end{abstract}

{\large
\section{Introduction}
Patterns of roots of a polynomial under differentiation as well as  comparison  between the asymptotic behaviors of the zero set of a polynomial  and the zero set of its derivatives have already been considered since the nineteenth century, but their study  has made progresses in the last decades notably with the works of B. Hanin \cite{ha1}.  Following a conjecture of Pemantle and Rivin \cite{Pemantle}, Kabluchko and Seidel \cite{Kab}, then O'Rourke and Williams \cite{ORou}, determined the asymptotic fluctuations of the critical point nearest to a given root. They provided a fine stochastic analysis of the local situation.

In this paper we restrict ourselves to the cases where all the roots of the polyomials are on the real line.
 
For $n \geq 2$, let us denote by $(\lambda_j)_{1 \leq j \leq n}$ the roots of a random polynomial $P_n$ such that  the corresponding empirical measure
$$ \mu_n := \frac {1}{n}\sum_{j=1}^{n} \delta_{\lambda_j}$$
tends to a limiting probability measure $\mu$ when $n$ tends to infinity, and for $1 \leq k \leq n-1$, let $\mu_{n}^k$ be the empirical measure attached to the  roots set of the $k$-th derivative of $P_n$.
In several interesting random contexts, for a fixed integer $k$,  $\mu_{n}^k$ also tends to the same measure $\mu$.  If $k=o(n)$  tends to infinity, this is true when the roots $\lambda_j$ are real (it is still an open question for complex roots, see \cite{MV24} and \cite{MVu} for partial results). However, if the differentiation is repeated $k=\lfloor n \tau \rfloor$ times, where $\tau \in (0,1)$ is viewed as a ``time" parameter, some  non-trivial macroscopic dynamics of root sets appears,
and for real roots, it has been proven that  the empirical measure $\mu_{n}^k$  converges to a well-defined measure $\mu_{[\tau]}$. 

\subsection{Some recent results on dynamics induced by derivation}

In 2019, Steinerberger \cite{Stein1} derived the following partial  differential equation to describe the asymptotic evolution of  root sets of random polynomials on the real axis, when their degree $n$ tends to infinity:
$$\partial_t u = \frac{1}{\pi}  \partial_x \left(\arctan \left(\frac{u}{Hu}\right)\right)$$
where $u$ is the (regular) limit density of the roots set and $Hu$ is the Hilbert transform of $u$.  He provided a  rather informal construction  of his PDE: following the classical interpretation, via logarithmic derivation, of a critical point of $P_n$ as an equilibrium of repulsion-attraction forces from the roots of $P_n$, he divided them into a uniform local near field and an averaged  far field  estimated via a Cauchy-Stieljes integral, hence a Hilbert transform of the density.

In the case when all roots are real, several autors developped  a connection between
repeated differentiation and free probability: $\mu_{[t]}$, up to rescaling, coincides with the free additive self-convolution power
and satisfy the same PDE, which appeared in
a paper of Shlyakhtenko and Tao \cite{tao}. This property has been rigorously proven in \cite{HKab21}.
 Another  proof due to  Arizmendi, Garza-Vargas and Perales [\cite{AVP},
Theorem 3.7], uses finite free probability, a subject developed in the papers of Marcus \cite{Marcus},
Marcus, Spielman, Srivastava \cite{MSS}, Gorin and Marcus \cite{Gorin-Marcus}. In the proof of \cite{AVP}, 
the $k$-th derivative of $P_n$ is essentially identified with the finite free multiplicative convolution of $P_n$ and the polynomial $x^k (1-x)^{n-k}$ ; then, thanks to their  Theorem 1.4, finite free probability converges to the usual free probability, and  the considered empirical distribution  converges to the free multiplicative convolution of $\mu$ and $ t\delta_0 +(1-t)\delta_1$. 

In the case when all roots are on the unit circle, an analogous behavior happens. The articles
 \cite{Gra,Kis, HKab21,Alazard} successively offered detailed analysis providing rigorous proof of crystallization under repeated differentiation. They show global regularity and exponential in time convergence to uniform density. In \cite{Kis} the global in time control follows from the analysis of the propagation of errors equation, with  nonlinear fractional diffusion inspired by similar developments in the modelisation of collective motions.  In  \cite{HKab21}, an ''explicit`` strategy of proof relies on complex analysis. Then, Kabluchko in \cite{Kab21} completed this analysis by considering a natural isomorphism between trigonometric polynomials having all their roots on the unit circle and polynomials with complex coefficients $a_j$ such that $a_{n-j}=\overline{a_j}$  for $0 \leq j \leq n$. 
A recent paper of  Arizmendi, Fujie and Ueda \cite{AFU} gives new proofs of the above cited
results by Kabluchko, using purely combinatorial tools. Their identities on sums over partitions provide results on limit theorems for finite free convolutions and their connection to free probability theory. 

\subsection{Eigenvalues of minors of unitarily invariant matrices}
During the last decades, the point process limits of the random matrix spectra were described via the limiting
joint densities, usually relying on some algebraic structure of the finite models: see the
monographs by Mehta \cite{Mehta}, Anderson, Guionnet and Zeitouni \cite{AGZ09} and Forrester \cite{Forrester}
for an overview of the classical results.
Dumitriu and Edelman \cite{DE02} constructed tridiagonal random matrix models with 
spectrum distributed as beta ensembles, one parameter extensions of classical random matrix
models. Edelman and Sutton \cite{ES07} observed that under the appropriate scaling, these 
tridiagonal matrix models behave like approximate versions of random stochastic operators, and
conjectured that scaling limits of beta ensembles can be described as the spectra of these
objects.

A large class of random matrix models are related to ensembles which are invariant in distribution by unitary or orthogonal conjugation.
Such ensembles correspond to matrices with the same distribution as 
$$M := U^* \operatorname{Diag}(\lambda_1, \dots, \lambda_n) U,$$
where $U$ is uniform on the group $U(n)$ or the group $O(n)$, and independent of $(\lambda_1, \dots, \lambda_n)$ where 
$\lambda_1, \dots, \lambda_n$ are possibly random real numbers. The characteristic polynomials of the successive top-left minors of $M$ form a random sequence of polynomials of decreasing degree, as the iterated derivatives of the polynomial with roots $\lambda_1, \dots, \lambda_n$: as explained below, these two settings can be unified.

\subsection{Unification}
We proceed to the unification of the two settings into a single one, depending on a parameter $\beta \in (0, \infty]$, $\beta = 1$ corresponding to real symmetric matrices and orthogonal conjugation,
$\beta = 2$ to Hermitian matrices and unitary  conjugation, and  $\beta = \infty$ to the sequence of derivations of a given polynomial. For this purpose, for $n \geq 1$, we define a sequence $(P_{n,m})_{0 \leq m \leq n}$ of random polynomials with real roots, 
where $P_{n,m}$ has degree $n-m$, the $n$ roots of $P_{n,0}$ are given at the beginning, and for $1 \leq m \leq n$, the roots of
$P_{n,m}$ are deduced from those of $P_{n,m-1}$ via a recurrence formula depending on $\beta$. 
 In \cite{GNV}, we discuss a similar setting for complex polynomials, but we restrict ourselves to real polynomials in the present article. 

The main theorem of this article shows, under a general boundedness assumption, that for all $\tau \in (0,1)$ and $\beta \in (0, \infty]$, the convergence of the empirical measure of the roots of $P_{n,0}$  towards a limiting measure $\mu$ implies the convergence of the empirical distribution of the zeros of $P_{n,\lfloor n \tau \rfloor}$ towards a limiting measure $\mu_{[\tau]}$ depending on $\mu$ and $\tau$ but not on $\beta$. In other words, for $\beta \in (0, \infty)$, we obtain the same dynamics for the limiting distribution of the roots as for $\beta = \infty$, i.e. for the case where iterated derivatives are considered.  
Hence, we can transfer the description of the dynamics obtained for  $\beta = \infty$ by the authors cited above to the case of finite $\beta$.

The precise setting which is considered, and the main Theorem \ref{th1}, are stated in Section \ref{realsetting}. The proof of Theorem \ref{th1} is given in Section \ref{prooftheorem1}. 

\subsection*{Notation}

It the article,  the notation $X = \mathcal{O}_{a,b,c} (Y)$ means that there exists $K_{a,b,c} > 0$, depending only on $a, b$ and $c$, such that one has always $|X| \leq K_{a,b,c} Y$. 

\section{The main setting} \label{realsetting}
For $n \geq 1$, and for $\beta \in (0, \infty]$, we define the polynomial $P_{n,0}$ by 
$$P_{n,0}(z) = \prod_{j=1}^n 
(z- \lambda^{(n,0)}_j)$$
where $(\lambda^{(n,0)}_j)_{1 \leq j \leq n}$ is a (possibly random)  nondecreasing sequence of real numbers. 

For $1 \leq m \leq n$, we consider 
a random vector 
$(\rho^{(n,m)}_j)_{1 \leq j \leq n+1-m}$ following the Dirichlet distribution with all parameters equal to $\beta/2$, the random vectors and the initial polynomial $P_{n,0}$ being all independent. In the limiting case $\beta = \infty$, we set $\rho^{(n,m)}_j := 1/(n+1-m)$. 

Then, for $1 \leq m \leq n$, we inductively define the polynomial
$$P_{n, m} (z) = \sum_{j=1}^{n+1-m}
\rho^{(n,m)}_j \prod_{1 \leq k \leq n+1 - m, k \neq j} (z- \lambda^{(n,m-1)}_j)$$
and the nondecreasing sequence 
$(\lambda^{(n,m)}_j)_{1 \leq j \leq n-m}$, such that 
$$ P_{n,m} (z) = \prod_{j=1}^{n-m}
(z- \lambda^{(n,m)}_j). 
$$
This construction is motivated by the following result:
\begin{proposition} \label{minors}
For $\beta = \infty$, $P_{n,m}$ is $(n-m)!/n!$ times the $m$-th derivative of $P_{n,0}$. 

For $\beta = 2$ (respectively $\beta = 1$), $(P_{n,m})_{0 \leq m \leq n-1}$ has the same joint distribution as the characteristic polynomials of the successive top-left $(n-m) \times (n-m)$ minors of the $n \times n$ matrix $M_n$, for $0 \leq m \leq n-1$, with 
$$M_n =  U_n^* \operatorname{Diag}(( \lambda_j^{(n,0)})_{1 \leq j \leq n} ) U_n,$$
 $U_n$ being independent of $( \lambda_j^{(n,0)})_{1 \leq j \leq n} )$ and Haar-distributed on the unitary group $U(n)$ (respectively the orthogonal group $O(n)$). 

\end{proposition}
\begin{proof}
The case $\beta = \infty$ is immediate. Let us now assume $\beta = 2$, and for $1 \leq p \leq n$, let us denote by $M_{n,p}$ the top-left $p \times p$ minor of $M_n$. 
For $0 \leq m \leq n-1$, the joint distribution of $M_{n,n-m}$ and the characteristic polynomial  of $M_{n, n-r}$
for $0 \leq r \leq m$, which are all measurable functions of $M_n$, is invariant by conjugation with any unitary matrix which is independent of $M_n$. In particular, this joint distribution is invariant by conjugation with 
$\operatorname{Diag} (U_{n-m}, I_m)$ if 
$U_{n-m}$ is Haar-distributed on $U(n-m)$ and independent of $M_n$. Now, conjugation by $\operatorname{Diag} (U_{n-m}, I_m)$ does not change the characteristic polynomial of $M_{n, n-r}$
for $0 \leq r \leq m$, so the joint distribution above is the same as the joint distribution
of $ U_{n-m} ^*  M_{n,n-m}   U_{n-m} $ and the characteristic polynomial of $M_{n, n-r}$
for $0 \leq r \leq m$. 
By independence of $U_{n-m} $ and $M_n$, and the fact that $M_{n,n-m}$ is Hermitian, and then $M_{n,n-m}$ is the unitary conjugate of the diagonal matrix $D_{n,n-m}$ given by the nondecreasing sequence of the eigenvalues of $M_{n,n-m}$, 
we get the same joint distribution 
as  $ U_{n-m} ^*  D_{n,n-m}   U_{n-m} $ and the characteristic polynomials of $M_{n, n-r}$ for $0 \leq r \leq m$. 

We deduce that conditionally on the characteristic polynomials $M_{n, n-r}$ for $0 \leq r \leq m$, 
the joint distribution of the characteristic polynomials of the minors of $M_{n,n-m}$ (which are 
$M_{n, n-r}$ for $m \leq r \leq n-1$)
is the same as the joint distribution of the characteristic polynomials of the minors of 
$ U_{n-m} ^*  D_{n,n-m}   U_{n-m} $, $D_{n,n-m}$ being the diagonal matrix with the same characteristic polynomial as $M_{n,n-m}$. We deduce that the sequence of characteristic polynomials of $M_{n,n-m}$ for $0 \leq m\leq n-1$ is an inhomogeneous Markov process, the transitions depending only on the size of the matrices which are involved. Since $(P_{n,m})_{1 \leq m \leq n-1}$ has, by construction, a similar inhomogeneous Markov structure, it is enough to check that $P_{n,n-1}$ has, for deterministic 
$(\lambda^{(n,0)}_j)_{1 \leq j \leq n}$, the same distribution as the characteristic polynomial of 
$M_{n,n-1}$. 

The characteristic polynomial of $M_{n,n-1}$ is the cofactor of indices $n,n$ of the matrix 
$z I_n - M_n$, and then it can be written as the $n,n$ entry of the matrix 
$$ \operatorname{det} ( z I_n - M_n )   (z I_n - M_n )^{-1} \in \mathcal{M}_n ( \mathbb{C}(z)),$$
and then it is equal to 
\begin{align*}
& \prod_{j=1}^n (z - \lambda^{(n,0)}_j) 
[   U_n^{*}   \operatorname{Diag}((  (z- \lambda_j^{(n,0)})^{-1} )_{1 \leq j \leq n} ) U_n]_{n,n}
 \\ & =  \prod_{j=1}^n (z - \lambda^{(n,0)}_j)  \sum_{k=1}^n (U^*_n)_{n,k}
 (z- \lambda_k^{(n,0)})^{-1} (U_n)_{k,n} 
 \\ & =  \sum_{k=1}^n |(U_n)_{k,n}|^2  \prod_{1 \leq j \leq n, j \neq k} (z - \lambda^{(n,0)}_j)
\end{align*}
This completes the proof of Proposition \ref{minors} for $\beta =2$, since  $((U_n)_{k,n})_{1 \leq k \leq n}$ is uniformly distributed on the unit sphere of $\mathbb{C}^n$, and then 
$(|(U_n)_{k,n}|^2)_{1 \leq k \leq n}$ is Dirichlet distributed, with all parameters equal to $1$. 

The proof for $\beta = 1$ is similar, after replacing Hermitian matrices by real symmetric matrices and unitary matrices by orthogonal matrices. In this case, we get that $((U_n)_{k,n})_{1 \leq k \leq n}$ is uniformly distributed on the unit sphere of $\mathbb{R}^n$, which implies that $(|(U_n)_{k,n}|^2)_{1 \leq k \leq n}$ is Dirichlet distributed, with all parameters equal to $1/2$.

\end{proof} 

The main result of the article is the following: 
\begin{theorem} \label{th1}
We fix $\tau \in (0,1)$, and for all $n \geq 1$, we consider 
$m = \lfloor n \tau \rfloor$, which implies $0 \leq m \leq n-1$. 

We assume that for some $A > 0$, the zeros of 
$P_{n,0}$ are in the interval $[-A, A]$ for all $n \geq 1$, and that the empirical measure
$$\mu_{n,0} := \frac{1}{n} \sum_{j=1}^{n}\delta_{ \lambda^{(n,0)}_{j}}$$
almost surely converges to a limiting measure $\mu$, necessarily supported in $[-A,A]$. 
Then, the empirical distribution $\mu_{n,m}$  of the zeros of $P_{n,m}$, given by
$$\mu_{n,m} := \frac{1}{n-m} \sum_{j=1}^{n-m} \delta_{\lambda^{(n,m)}_{j}},$$
almost surely converges to the same limit as the empirical distribution of the zeros of the $m$-th derivative of $P_{n,0}$, i.e. this limit exists and
is independent of $\beta \in (0, \infty]$. 
\end{theorem}

The proof of this theorem is given in the following section. 
\section{Proof of theorem \ref{th1}} \label{prooftheorem1}
The convergence of the empirical distribution of the zeros of the $m$-th derivative of $P_{n,0}$, corresponding to the case $\beta = \infty$, is proven in \cite{HKab21,HoS,Kab,Kis,tao,AFU,AJ, AVP}. From now, we assume that $\beta$ is finite. 

We fix $z$ larger that all roots of  $P_{n,0}$, and then larger than all roots of $P_{n,m}$, $0 \leq m \leq n$, by interlacing property. 
In this case, all terms and factors involved in the expressions above are positive. 

For  $n \geq 1$, $0 \leq k \leq n$ and  $0 \leq r \leq n-k$, let us  denote by $ P_{n,k}^r(z) $ the $r$-th derivative of $P_{n,k}$ at $z$, renormalized in such a way that the leading coefficient of the polynomial is equal to $1$. 

For some $m$, $0 \leq m \leq n$, let us compare $P_{n,m}^0(z):=P_{n,m}(z)$ with 
$ P_{n,0}^m(z) :=P_{n,0}^{(m)}(z) (n-m)!/ n! $, the renormalized $m$-th derivative of $P_{n,0}$ at $z$.
For this purpose, we compare, for $0 \leq k \leq m-1$, 
$ P_{n,k}^{m-k}(z)$
with $ P_{n,k+1}^{m-k-1}(z)$, and  we construct a triangular sequence of polynomials 
\begin{align*}
P_{n,0}  \;  &   \; P_{n,1} & ... & \; & P_{n,k} \;  & \;  P_{n,k+1} &  ... \; & \; P_{n,m-1} &  P_{n,m} \\
P_{n,0}^1  \; &  \;  P_{n,1}^1 & ... & \; & P_{n,k}^1 \; & \; P_{n,k+1}^1 &  ... \; & \; P_{n,m}^1 \\
. \\
. \\
. & \;   & ... & \; & \; & \; P_{n,k+1}^{m-k-1} \\
.  & \;   & ... & \; & \;  P_{n,k}^{m-k} \\
. \\
. \\
P_{n,0}^m  \\
\end{align*}
We observe that 
for all $\lambda_1, \dots, \lambda_p \in \mathbb{R}$ and $0 \leq q \leq p$,
\begin{align*}\frac{d^q}{dz^q} \prod_{j=1}^p (z - \lambda_j) 
& = \sum_{1 \leq j_1 \neq \dots \neq j_q \leq p} \prod_{1 \leq j \leq p, 
j \neq j_1, \dots, j_q} (z - \lambda_j) 
\\ & = q! \sum_{J \in \{1, \dots, p\}, |J| = q} \prod_{1 \leq j \leq p, 
j \notin J} (z - \lambda_j)
\\ & = q! \sum_{J^c \in \{1, \dots, p\}, |J^c| = p-q} \prod_{1 \leq j \leq p, 
j \in J^c} (z - \lambda_j)
\\ & = \frac{q!}{(p-q)!} 
\sum_{1 \leq j_1 \neq \dots \neq j_{p-q} \leq p}
\prod_{1 \leq \ell \leq p-q} (z - \lambda_{j_{\ell}}) 
\end{align*} 

Hence, 

\begin{equation}
 P_{n,k}^{m-k}(z):=\frac{(n-m)! P_{n,k}^{(m-k)}(z)}{ (n-k)!}  = \frac{(m-k)!}{(n-k)!} 
 \sum_{1 \leq j_1 \neq j_2 \neq \dots \neq j_{n-m} \leq n-k} \, \prod_{\ell=1}^{n-m}  (z- \lambda^{(n,k)}_{j_\ell}).
\label{Pnk}
\end{equation}
Moveover, since $P_{n, k+1} (z) = \sum_{j=1}^{n-k}
\rho^{(n,k+1)}_j  \prod_{1 \leq r \leq n-k, r \neq j} (z- \lambda^{(n,k)}_{r})$, we have
\begin{align*}  P_{n,k+1}^{m-k-1} (z) & = \frac{(m-k-1)!}{(n-k-1)!} 
\sum_{j=1}^{n-k} \rho_j^{(n,k+1)} \, \sum_{1 \leq j_1 \neq j_2 \neq \dots \neq j_{n-m} \neq j \leq n-k}
\prod_{\ell=1}^{n-m}  (z- \lambda^{(n,k)}_{j_\ell}),
\\ &  = \sum_{j=1}^{n-k} X_j \rho_j^{(n,k+1)}
\end{align*}
where for $1 \leq j \leq n-k$,
$$X_j = \frac{(m-k-1)!}{(n-k-1)!}
\, \sum_{1 \leq j_1 \neq j_2 \neq \dots \neq j_{n-m} \neq j \leq n-k}
\prod_{\ell=1}^{n-m}  (z- \lambda^{(n,k)}_{j_\ell}).
$$
The expression \eqref{Pnk} can be 
modified by introducing the choice of an extra index 
$j \in \{1, \dots, n-k\}$, which should be 
different from $j_1, \dots, j_{n-m}$. For given $j_1, \dots, j_{n-m}$, the number of possible choices for $j$ is $(n-k)-(n-m) = m-k$. Hence,
\begin{align*}
 P_{n,k}^{m-k}(z) & = \frac{(m-k)!}{(n-k)!} 
 \sum_{1 \leq j_1 \neq j_2 \neq \dots \neq j_{n-m} \leq n-k} 
\;  \sum_{1 \leq j \leq n-k, 
 j \neq j_1, \dots, j_{n-k}}
 \frac{1}{m-k}
 \prod_{\ell=1}^{n-m}  (z- \lambda^{(n,k)}_{j_\ell}).
 \\ & = \frac{(m-k)!}{(n-k)!}
 \sum_{j=1}^{n-k} 
 \sum_{1 \leq j_1 \neq j_2 \neq \dots \neq j_{n-m} \neq j \leq n-k}
\frac{1}{m-k} \prod_{\ell=1}^{n-m}  (z- \lambda^{(n,k)}_{j_\ell}),
\\ & = \frac{1}{n-k} \sum_{j=1}^{n-k} X_j.
\end{align*}
We notice that the expectation of the random variables 
$\rho_j^{(n,k+1)}$ is equal to $\frac{1}{n-k}$.\\
We now have the following result: 
\begin{proposition} \label{prop2}
For $A  \geq \max_{1 \leq j \leq n} |\lambda^{(n,0)}_j|$, $z > A +1$, and   $\delta > 0$, one has
$$\mathbb{P} 
\left( \left| \log \left( P_{n,k+1}^{m-k-1} (z)   \right) - \log \left(  P_{n,k}^{m-k}(z) \right)\right| \geq \delta  \right)
= \mathcal{O}_{\beta, A, \delta} ((n-k)^{-3}).$$
\end{proposition}
\begin{proof}
With the notation above, we have to show 
$$\mathbb{P} \left( 
\frac{\sum_{j=1}^{n-k} X_j (n-k) \rho^{(n, k+1)}_j}{\sum_{j=1}^{n-k} X_j  } \notin (e^{-\delta}, e^{\delta}) \right) = \mathcal{O}_{\beta, A, \delta}  ((n-k)^{-3}).$$
It is sufficient to prove the 
same result for the conditional 
probability given the values of 
$(\lambda_j^{(n,k)})_{1 \leq j \leq n-k} $, where, by the properties of  interlacing  satisfied by the zeros of the polynomials involved in this discussion, $|\lambda_j^{(n,k)}| \leq A$ for all $j \in \{1, \dots, n-k\}$. 
It is then enough to prove the following: 
for all deterministic families $(\lambda_j)_{1 \leq j \leq n-k}$ of elements in $[-A,A]$,  
for
$$Y_j = \sum_{1 \leq j_1 \neq \dots \neq j_{n-m} \neq j \leq n-k} \prod_{\ell = 1}^{n-m} (z - \lambda_{j_{\ell}}),$$
$$\mathbb{P} \left( 
\frac{\sum_{j=1}^{n-k} Y_j (n-k) \rho^{(n, k+1)}_j}{\sum_{j=1}^{n-k} Y_j  } \notin (e^{-\delta}, e^{\delta}) \right) = \mathcal{O}_{\beta, A, \delta}  ((n-k)^{-3}).$$
Since $(\rho^{(n,k+1)}_j)_{1 \leq j \leq n-k}$ is Dirichlet distributed with all parameters equal to $\beta/2$, it has the same distribution as 
$$ \left( \frac{\gamma_j}{ \sum_{\ell=1}^{n-k} \gamma_{\ell} } \right)_{1 \leq j \leq n-k}$$
where $(\gamma_j)_{1 \leq j \leq n-k}$ are i.i.d. Gamma variables with parameter $\beta/2$. 
By large deviation properties 
satisfied by the Gamma variables, 
we have 
$$\mathbb{P} \left( 
\frac{1}{(\beta/2) (n-k)} \sum_{\ell=1}^{n-k} \gamma_{\ell}
\notin (e^{-\delta/2}, e^{\delta/2}) 
\right) = \mathcal{O}_{\beta, \delta} (e^{-c_{\beta,\delta} (n-k)}) = \mathcal{O}_{\beta, A, \delta}  ((n-k)^{-3}).
$$
where $c_{\beta, \delta} > 0$ depends only on $\beta$ and $\delta$. 
If this event does not occur, 
we can then write 
$$\frac{(\beta/2)(n-k) \rho_j^{(n,k+1)} }{ 
\gamma_j} \in (e^{-\delta/2}, e^{\delta/2} ).$$
Since $Y_j$ is a sum of positive terms (recall that 
$z > A+1$ and $|\lambda_j| \leq A$
for $1 \leq j \leq n-k$),
it is sufficient to prove 
$$\mathbb{P} \left( 
\frac{\sum_{j=1}^{n-k} Y_j \gamma_j}{(\beta/2) \sum_{j=1}^{n-k} Y_j  } \notin (e^{-\delta/2}, e^{\delta/2}) \right) = \mathcal{O}_{\beta, A, \delta}  ((n-k)^{-3}),$$
i.e. 
$$\mathbb{P} \left( 
\frac{\sum_{j=1}^{n-k} Y_j \widetilde{\gamma}_j}{(\beta/2) \sum_{j=1}^{n-k} Y_j  } \notin (e^{-\delta/2} - 1, e^{\delta/2} - 1) \right) = \mathcal{O}_{\beta, A, \delta}  ((n-k)^{-3}),$$
where 
$$\widetilde{\gamma}_j := 
\gamma_j - \frac{\beta}{2} = 
\gamma_j - \mathbb{E} [\gamma_j].$$
It is then enough to show, 
for $\varepsilon > 0$, 
$$\mathbb{P} 
\left( \left| \sum_{j=1}^{n-k} 
R_j \widetilde{\gamma}_j \right| > \varepsilon \right)  = \mathcal{O}_{\beta, A, \varepsilon}  ((n-k)^{-3}),$$
for 
$$R_j = \frac{Y_j}{ \sum_{\ell = 1}^{n-k} Y_{\ell} }.$$
Now, for $1 \leq j, j' \leq n-k$, $Y_{j}$ and $Y_{j'}$ are sums of products of positive factors, and one can go from the expression of $Y_{j}$ to the expression of $Y_{j'}$ by changing the factor $z - \lambda_{j'}$ by $z - \lambda_{j} $ each time it appears. This change multiply each term by a ratio
$(z-\lambda)/(z - \lambda')$
for $z > A+1$ and $\lambda, \lambda' \in [-A,A]$. This ratio is 
always between $1/(2A+1)$ and $2A+1$. Hence, 
$$\frac{1}{2A+1} \leq \frac{Y_j}{Y_{j'}} \leq 2A+1.$$
We deduce 
$$R_j = \frac{Y_j}{ \sum_{\ell = 1}^{n-k} Y_{\ell} }
\leq \frac{Y_j}{ \sum_{\ell = 1}^{n-k} Y_{j}/(2A+1)}
= \frac{2A+1}{n-k}
= \mathcal{O}_A ((n-k)^{-1}).
$$
Applying Markov's inequality to the sixth power, it is enough to show 
$$\mathbb{E} \left[ 
\left(\sum_{j=1}^{n-k}  
R_j \widetilde{\gamma}_j \right)^6 \right] = \mathcal{O}_{\beta,A} ((n-k)^{-3}).$$
The left-hand side can be written as 
$$\sum_{1 \leq j_1, \dots, j_6 \leq n-k} 
R_{j_1} R_{j_2} \dots R_{j_6} 
\mathbb{E} [\widetilde{\gamma}_{j_1}
\widetilde{\gamma}_{j_2} \dots 
\widetilde{\gamma}_{j_6} ]
= \mathcal{O}_{\beta,A} ((n-k)^{-6})\sum_{1 \leq j_1, \dots, j_6 \leq n-k} 
\mathbb{E} [\widetilde{\gamma}_{j_1}
\widetilde{\gamma}_{j_2} \dots 
\widetilde{\gamma}_{j_6} ].$$
The last expectation is always $\mathcal{O}_{\beta} (1)$. Moreover, by the fact that the variables $(\widetilde{\gamma}_{j})_{1 \leq j \leq n-k}$ are independent and centered, the expectation vanishes 
as soon as one of the six indices 
is different from all the others.
Hence, 
$$\mathbb{E} \left[ 
\left(\sum_{j=1}^{n-k}  
R_j \widetilde{\gamma}_j \right)^6 \right] 
= \mathcal{O}_{\beta,A} ( \mathcal{N} (n-k)^{-6})$$
where $\mathcal{N}$
 is the number of $6$-tuples
of indices between $1$ and $n-k$, such that none of the possible indices appears exactly once. 
For these $6$-uples, each index which appears is involved at least twice, so there are at most three different indices. The number of choices 
for these indices is at most $(n-k)^3$, 
and for each of these choices, there is a bounded number of orderings of the indices.  
Hence, $\mathcal{N}  = \mathcal{O}((n-k)^3)$ and 
$$\mathbb{E} \left[ 
\left(\sum_{j=1}^{n-k}  
R_j \widetilde{\gamma}_j \right)^6 \right] 
= \mathcal{O}_{\beta,A} (  (n-k)^{-3}),$$
which completes the proof of the proposition. 
 
From a union bound for the $m$ events corresponding to each value of $k \in \{0, 1, \dots, m-1\}$, and from the fact that 
$$P_{n,0}^m(z) = \frac{(n-m)! P_{n,0}^{(m)}(z)}{ n!},$$ 
we immediately deduce the following: 
\begin{corollary} \label{corollary}
Let $A \geq \max_{1 \leq j \leq n} |\lambda^{(n,0)}_j|$, let us assume $z > A +1$, and let $\delta > 0$. Then, 
$$\mathbb{P} 
\left( \left| \log \left( P_{n,m} (z)  \right) - \log \left( \frac{(n-m)! P_{n,0}^{(m)}(z)}{ n!}\right) \geq m \delta \right| \right)
= \mathcal{O}_{\beta, A, \delta} (m (n+1-m)^{-3}).$$

\end{corollary}

Let us now complete the proof of Theorem \ref{th1}.

By Corollary \ref{corollary}, 
for $n \geq 1$ larger than 
a value depending only on $\tau$, $m = \lfloor n \tau \rfloor$, 
and for $z > A+1$, we have, with probability $1 - \mathcal{O}_{\beta, A, \delta, \tau} (n^{-2})$, 

$$\left| \log P_{n,m}(z) - \log 
\left((n-m)! P_{n,0}^{(m)}(z) / n!  \right)\right| \leq m \delta 
\leq  n \delta. $$
Hence, for all $z > A+1$, $\delta > 0$, we have almost surely 

$$\underset{n \rightarrow \infty}{\lim \sup} \left|\frac{1}{n} \log P_{n,m}(z) - \frac{1}{n} \log 
\left((n-m)! P_{n,0}^{(m)}(z) / n!  \right)\right| \leq \delta.  
$$
Since $\delta > 0$ is arbitrary, 
for all $z  > A+1$, we have almost surely
$$\frac{1}{n} \log P_{n,m}(z) - \frac{1}{n} \log 
\left((n-m)! P_{n,0}^{(m)}(z) / n!  \right) \underset{n \rightarrow \infty}{\longrightarrow} 0. $$
Since $n$ and $n-m$ have the same order of magnitude, this convergence is equivalent to 
$$\int_{\mathbb{R}} 
\log (z - \lambda) d \mu_{n,m} (\lambda) - \int_{\mathbb{R}} 
\log (z - \lambda) d \mu'_{n,m} (\lambda) \underset{n \rightarrow \infty}{\longrightarrow} 0$$
where $\mu_{n,m}$ is the empirical distribution of the zeros of $P_{n,m}$, 
$$\mu_{n,m} := \frac{1}{n-m} \sum_{j=1}^{n-m} \lambda^{(n,m)}_{j}$$
and $\mu'_{n,m}$ is the empirical distribution of the zeros of 
the $m$-th derivative of $P_{n,0}$. 
From results related to the evolution of the zeros of the derivatives of random polynomials see e.g. \cite{AJ},
we deduce that since $\mu_{n,0}$ converges to a deterministic measure $\mu$ when $n \rightarrow \infty$, $\mu'_{n,m}$ converges to a measure $\mu_{[\tau]}$ depending only on $\mu$ and $\tau$. 
Since for $z > A+1$, $\lambda \mapsto \log (z -\lambda)$ is continuous and bounded on the interval $[-A,A]$, which contains the support of all measures involved here, we get 
$$\int_{\mathbb{R}} 
\log (z - \lambda) d \mu'_{n,m} (\lambda) \underset{n \rightarrow \infty}{\longrightarrow} 
\int_{\mathbb{R}} 
\log (z - \lambda) d\mu_{[\tau]} (\lambda).$$
Hence, almost surely, 
$$\int_{\mathbb{R}} 
\log (z - \lambda) d \mu_{n,m} (\lambda) \underset{n \rightarrow \infty}{\longrightarrow} 
\int_{\mathbb{R}} 
\log (z - \lambda) d\mu_{[\tau]} (\lambda).$$
Almost surely, this convergence occurs for all $z \in \mathbb{Q}$, 
$z > A+1$. 
Since the measures $(\mu_{n,m})_{n \geq 1, m = \lfloor n \tau \rfloor}$ are supported in $[-A,A]$, they form a random, tight family of probability measures. 
Let $\nu$ be a possibly random limit point of this sequence. 
We have, along a random subsequence, for all $z > A + 1$, 
$$\int_{\mathbb{R}} 
\log (z - \lambda) d \mu_{n,m} (\lambda) \underset{n \rightarrow \infty}{\longrightarrow} 
\int_{\mathbb{R}} 
\log (z - \lambda) d \nu (\lambda)$$
and then almost surely, for all $z \in \mathbb{Q} \cap (A+1, \infty)$,  
$$\int_{\mathbb{R}} 
\log (z - \lambda) d \nu (\lambda)
= \int_{\mathbb{R}} 
\log (z - \lambda) d \mu_{[\tau]} (\lambda).$$
The equality extends to all reals $z > A+1$ by continuity. 
Then, by taking the derivatives, 
the Stieltjes transforms of $\mu_{[\tau]}$ and $\nu$ almost surely coincide on the interval $(A+1, \infty)$. 
and then on $\mathbb{C} \backslash [-A,A]$ by analytic continuation. 
This is enough to show that 
$\mu_{[\tau]} = \nu$ almost surely, i.e. any random limit point of 
$(\mu_{n,m})_{n \geq 1, m = \lfloor n \tau \rfloor}$ is almost surely equal to $\mu_{[\tau]}$. 
On the event $\mathcal{E}$ where $(\mu_{n,m})_{n \geq 1, m = \lfloor n \tau \rfloor}$ does not converge to $\mu_{[\tau]}$, one can extract a random subsequence of $(\mu_{n,m})_{n \geq 1, m = \lfloor n \tau \rfloor}$ such that the Prokhorov distance to $\mu_{[\tau]}$ is bounded from below, 
and then a further random subsequence which converges, the limit $\nu$ being necessarily different from $\mu_{[\tau]}$. Since we have proven that $\nu = \mu_{[\tau]}$ almost surely, the event $\mathcal{E}$ necessarily has probability zero, i.e. $(\mu_{n,m})_{n \geq 1, m = \lfloor n\tau \rfloor}$ almost surely converges to $\mu_{[\tau]}$. 
\end{proof}


\begin{thebibliography}{10}
\expandafter\ifx\csname url\endcsname\relax
  \def\url#1{\texttt{#1}}\fi
\expandafter\ifx\csname doi\endcsname\relax
  \def\doi#1{\burlalt{doi:#1}{http://dx.doi.org/#1}}\fi
\expandafter\ifx\csname urlprefix\endcsname\relax\def\urlprefix{URL }\fi
\expandafter\ifx\csname href\endcsname\relax
  \def\href#1#2{#2}\fi
\expandafter\ifx\csname burlalt\endcsname\relax
  \def\burlalt#1#2{\href{#2}{#1}}\fi

\bibitem{Alazard}
T.~Alazard, O.~Lazar, and Q.-H. Nguyen.
\newblock On the dynamics of the roots of polynomials under differentiation,
  2022, \burlalt{2104.06921}{http://arxiv.org/abs/2104.06921}.

\bibitem{AGZ09}
G.~W. Anderson, A.~Guionnet, E.~Lyon, and O.~Zeitouni.
\newblock {\em An Introduction to Random Matrices}.
\newblock 2009.
\newblock \urlprefix\url{https://cims.nyu.edu/~zeitouni/cupbook.pdf}.

\bibitem{AFU}
O.~Arizmendi, K.~Fujie, and Y.~Ueda.
\newblock New combinatorial identity for the set of partitions and limit
  theorems in finite free probability theory, 2023,
  \burlalt{2303.01790}{http://arxiv.org/abs/2303.01790}.

\bibitem{AVP}
O.~Arizmendi, J.~Garza-Vargas, and D.~Perales.
\newblock Finite free cumulants: multiplicative convolutions, genus expansion
  and infinitesimal distributions.
\newblock {\em Trans. Amer. Math. Soc.}, 376(6):4383--4420, 2023.
\newblock \doi{10.1090/tran/8884}.

\bibitem{AJ}
O.~Arizmendi and S.~G.~G. Johnston.
\newblock Free probability via entropic optimal transport, 2024,
  \burlalt{2309.12196}{http://arxiv.org/abs/2309.12196}.

\bibitem{AssiotisNajnudel}
T.~Assiotis and J.~Najnudel.
\newblock The boundary of the orbital beta process.
\newblock {\em Mosc. Math. J.}, 21(4):659--694, 2021.
\newblock \doi{10.17323/1609-4514-2021-21-4-659-694}.

\bibitem{BEY}
P.~Bourgade, L.~Erd\H{o}s, and H.-T. Yau.
\newblock Universality of general {$\beta$}-ensembles.
\newblock {\em Duke Math. J.}, 163(6):1127--1190, 2014.
\newblock \doi{10.1215/00127094-2649752}.

\bibitem{BNR}
P.~Bourgade, A.~Nikeghbali, and A.~Rouault.
\newblock Circular {J}acobi ensembles and deformed {V}erblunsky coefficients.
\newblock {\em Int. Math. Res. Not. IMRN}, (23):4357--4394, 2009.
\newblock \doi{10.1093/imrn/rnp092}.

\bibitem{Cuenca}
C.~Cuenca.
\newblock Universal behavior of the corners of orbital beta processes.
\newblock {\em Int. Math. Res. Not. IMRN}, (19):14761--14813, 2021.
\newblock \doi{10.1093/imrn/rnz226}.

\bibitem{DE02}
I.~Dumitriu and A.~Edelman.
\newblock Matrix models for beta ensembles.
\newblock {\em J. Math. Phys.}, 43(11):5830--5847, 2002.
\newblock \doi{10.1063/1.1507823}.

\bibitem{ES07}
A.~Edelman and B.~D. Sutton.
\newblock From random matrices to stochastic operators.
\newblock {\em J. Stat. Phys.}, 127(6):1121--1165, 2007.
\newblock \doi{10.1007/s10955-006-9226-4}.

\bibitem{Forrester}
P.~Forrester.
\newblock {\em Log-Gases and Random Matrices (LMS-34)}.
\newblock Princeton University Press, 2010.
\newblock \urlprefix\url{http://www.jstor.org/stable/j.ctt7t5vq}.

\bibitem{GNV}
A.~Galligo, J.~Najnudel, and T.~Vu.
\newblock {Anti-concentration applied to roots of randomized derivatives of
  polynomials}.
\newblock {\em Electronic Journal of Probability}, 29:1--20, 2024.
\newblock \doi{10.1214/24-EJP1180}.

\bibitem{Gorin-Marcus}
V.~Gorin and A.~W. Marcus.
\newblock Crystallization of random matrix orbits.
\newblock {\em Int. Math. Res. Not. IMRN}, (3):883--913, 2020.
\newblock \doi{10.1093/imrn/rny052}.

\bibitem{Gra}
R.~Granero-Belinch\'{o}n.
\newblock On a nonlocal differential equation describing roots of polynomials
  under differentiation.
\newblock {\em Commun. Math. Sci.}, 18(6):1643--1660, 2020.
\newblock \doi{10.4310/CMS.2020.v18.n6.a6}.

\bibitem{ha1}
B.~Hanin.
\newblock Correlations and pairing between zeros and critical points of
  {G}aussian random polynomials.
\newblock {\em Int. Math. Res. Not. IMRN}, (2):381--421, 2015.
\newblock \doi{10.1093/imrn/rnt192}.

\bibitem{HKab21}
J.~Hoskins and Z.~Kabluchko.
\newblock Dynamics of zeroes under repeated differentiation.
\newblock {\em Exp. Math.}, 32(4):573--599, 2023.
\newblock \doi{10.1080/10586458.2021.1980752}.

\bibitem{HoS}
J.~G. Hoskins and S.~Steinerberger.
\newblock A semicircle law for derivatives of random polynomials, 2020,
  \burlalt{2005.09809}{http://arxiv.org/abs/2005.09809}.

\bibitem{Kab21}
Z.~Kabluchko.
\newblock Repeated differentiation and free unitary poisson process, 2022,
  \burlalt{2112.14729}{http://arxiv.org/abs/2112.14729}.

\bibitem{Kab}
Z.~Kabluchko and H.~Seidel.
\newblock Distances between zeroes and critical points for random polynomials
  with i.i.d. zeroes.
\newblock {\em Electron. J. Probab.}, 24:Paper No. 34, 25, 2019.
\newblock \doi{10.1214/19-EJP295}.

\bibitem{KN}
R.~Killip and I.~Nenciu.
\newblock Matrix models for circular ensembles.
\newblock {\em Int. Math. Res. Not.}, (50):2665--2701, 2004.
\newblock \doi{10.1155/S1073792804141597}.

\bibitem{Kis}
A.~Kiselev and C.~Tan.
\newblock The flow of polynomial roots under differentiation, 2020,
  \burlalt{2012.09080}{http://arxiv.org/abs/2012.09080}.

\bibitem{Marcus}
A.~W. Marcus.
\newblock Polynomial convolutions and (finite) free probability, 2021,
  \burlalt{2108.07054}{http://arxiv.org/abs/2108.07054}.

\bibitem{MSS}
A.~W. Marcus, D.~A. Spielman, and N.~Srivastava.
\newblock Finite free convolutions of polynomials.
\newblock {\em Probab. Theory Related Fields}, 182(3-4):807--848, 2022.
\newblock \doi{10.1007/s00440-021-01105-w}.

\bibitem{Mehta}
M.~L. Mehta.
\newblock {\em Random Matrices}.
\newblock 3rd edition, 2004.

\bibitem{MV24}
M.~Michelen and X.-T. Vu.
\newblock {Almost sure behavior of the zeros of iterated derivatives of random
  polynomials}.
\newblock {\em Electronic Communications in Probability}, 29:1--10, 2024.
\newblock \doi{10.1214/24-ECP596}.

\bibitem{MVu}
M.~Michelen and X.-T. Vu.
\newblock Zeros of a growing number of derivatives of random polynomials with
  independent roots.
\newblock {\em Proc. Amer. Math. Soc.}, 152(6):2683--2696, 2024.

\bibitem{NV}
J.~Najnudel and B.~Vir\'{a}g.
\newblock The bead process for beta ensembles.
\newblock {\em Probab. Theory Related Fields}, 179(3-4):589--647, 2021.
\newblock \doi{10.1007/s00440-021-01034-8}.

\bibitem{Stein1}
S.~O'Rourke and S.~Steinerberger.
\newblock A nonlocal transport equation modeling complex roots of polynomials
  under differentiation.
\newblock {\em Proc. Amer. Math. Soc.}, 149(4):1581--1592, 2021.
\newblock \doi{10.1090/proc/15314}.

\bibitem{ORou}
S.~O'Rourke and N.~Williams.
\newblock On the local pairing behavior of critical points and roots of random
  polynomials.
\newblock {\em Electron. J. Probab.}, 25:Paper No. 100, 68, 2020.
\newblock \doi{10.1214/20-ejp499}.

\bibitem{Pemantle}
R.~Pemantle and I.~Rivin.
\newblock The distribution of zeros of the derivative of a random polynomial.
\newblock In {\em Advances in combinatorics}, pages 259--273. Springer,
  Heidelberg, 2013.

\bibitem{tao}
D.~Shlyakhtenko and T.~T.~W. an~appendix~by David~Jekel.
\newblock Fractional free convolution powers, 2021,
  \burlalt{2009.01882}{http://arxiv.org/abs/2009.01882}.

\end{thebibliography}
\end{document}